\documentclass[12pt]{article}
\usepackage{float}
\usepackage{color}
\usepackage{graphics}
\usepackage{graphicx}
\usepackage[all]{xy}
\usepackage[utf8]{inputenc}
\usepackage[english]{babel}

\usepackage[T1]{fontenc}

\usepackage[a4paper]{geometry}
\usepackage{amsmath}
\usepackage{amssymb}
\usepackage{amsthm}
\usepackage{amsfonts}

    \makeatletter
    \renewcommand
    {\section}{\@startsection{section}{1}{0mm}
    {2\baselineskip}{0.25\baselineskip}
    {\normalfont\large\slshape\bfseries}}
    \makeatother

\usepackage{titlesec}
\titlelabel{\thetitle.\quad}

\title{Mapping class group dynamics on $\mathrm{Aff}(\mathbb{C})$-characters.}
\author{Selim Ghazouani \\
selim.ghazouani@ens.fr}

\date{}

\newtheorem{theorem}{Theorem}
\newtheorem{lemme}[theorem]{Lemma}
\newtheorem{proposition}[theorem]{Proposition}
\newtheorem{definition}[theorem]{Definition}

\newtheorem{conjecture}[theorem]{Conjecture}
\newtheorem{corollary}[theorem]{Corollary}

\begin{document}

\maketitle

\begin{center}
DMA - Ecole Normale Supérieure 
\\ 45, rue d'Ulm 
\\ 75230 Paris Cedex 05 - France

\end{center}

\begin{abstract}

We prove that in genus greater than $2$, the mapping class group action on $\mathrm{Aff}(\mathbb{C})$-characters is ergodic. This implies that almost every representation $\pi_1 S \longrightarrow \mathrm{Aff}(\mathbb{C})$ is the holonomy of a branched affine structure on $S$, when $S$ is a closed orientable surface of genus $g \geq 2$.

\vspace{2mm}

\end{abstract}

\tableofcontents

\newpage

\section*{Introduction}

Let $\Gamma$ be the fundamental group of a compact orientable surface $S$ of genus $g\geq 2$. If $G$ is a finite dimensional reductive Lie group (typically  $G = \mathrm{PSL}(2,\mathbb{R})$ or $\mathrm{SU}(2)$), one can look at the character variety $\chi(\Gamma, G)$ which is defined to be the quotient $\mathrm{Hom}(\Gamma, G) // G$, in the sense of geometric invariant theory. The mapping class group of  $S$ acts on $\chi(\Gamma, G)$  by precomposition, the study of this action was popularized by Goldman in the early 80's. 
\noindent The most classical result in the field, by Goldman, is  that the action is ergodic for $G = \mathrm{SU}(2)$ (see \cite{Gold}). This result was extended  by Pickrell and Xia to the case where $G$ is compact, see \cite{PX}. In this paper we study the case $G = \mathrm{Aff}(\mathbb{C}) = \{ z \mapsto az + b \ | \ (a,b) \in \mathbb{C}^* \times \mathbb{C} \}$.  Since $\mathrm{Aff}(\mathbb{C}) $ is solvable, the tools from symplectic geometry developed in the reductive case do not apply in our setting. Moreover, the character variety is not defined, at least in the sense of geometric invariant theory. This last difficulty can be avoided by defining $\chi (\Gamma, \mathrm{Aff}(\mathbb{C}) ) $ to be the quotient of $\mathrm{Hom}(\Gamma, \mathrm{Aff}(\mathbb{C})) \setminus \{ \text{abelian} \ \text{representations}\}$ by the action of $G$ by conjugation (see Section \ref{charactervariety}).

 $\chi(\Gamma, \mathrm{Aff}(\mathbb{C}) )$ has a structure of fiber bundle. It comes from the isomorphism $\mathrm{Aff}(\mathbb{C}) \simeq \mathbb{C}^* \ltimes \mathbb{C}$, a representation $\rho : \Gamma \longrightarrow \mathrm{Aff}(\mathbb{C}) $ is the data of a linear part $\alpha : \Gamma \longrightarrow \mathbb{C}^*$ and a translation part $\lambda : \Gamma \longrightarrow \mathbb{C}$ ($\rho= (\alpha, \lambda) : \Gamma \rightarrow \mathbb{C}^* \ltimes \mathbb{C} $)  , where $\alpha$ is a group homomorphism and $\lambda$ is a cocyle relation twisted by $\alpha$. A point in the quotient space will be parametrized by an element in $\mathrm{H}^1(S, \mathbb{C}^*) \simeq (\mathbb{C}^*)^{2g}$ (the linear part) and an element in the projectivized space of $\mathrm{H}^1_{\alpha}(\Gamma, \mathbb{C}^*) \simeq \mathbb{CP}^{2g-3}$ (the translation part), and this parametrization gives the fiber bundle structure.

In the case where $G = \mathbb{C}$ (the simplest non reductive case), the character variety is $ \mathrm{H}^1(S, \mathbb{C}) \simeq \mathbb{C}^{2g}$. The action of the mapping class group on  $\mathrm{H}^1(S, \mathbb{C})$ (which happen to factor through the linear action of $\mathrm{Sp}(2g, \mathbb{Z})$ on $\mathbb{C}^{2g}$) has an invariant non constant continuous function, $\omega \longmapsto \omega \wedge \overline{\omega} \in  \mathrm{H}^2(S, \mathbb{R}) \simeq \mathbb{R}$. Hence this action is not ergodic. (A careful study of this action has been carried out by M.Kapovich in \cite{HK}). The main result of our paper is   

\begin{theorem}

The mapping class group action on $\chi(\Gamma, \mathrm{Aff}(\mathbb{C}) )$ is ergodic.

\end{theorem}

 The mapping class group action preserves this fiber bundle structure, and to prove the theorem we first prove that the induced action on the base is ergodic. Then we observe that the Torelli group stabilizes globally the fibers, and we prove that its action is ergodic in almost every fiber.

\begin{itemize}

\item The action on $\mathrm{H}^1(S, \mathbb{C}^*)$ is actually the linear diagonal action of $\mathrm{Sp}(2g, \mathbb{Z})$ on $\mathbb{R}^{2g} \times (\mathbb{R}/\mathbb{Z})^{2g} $. Moore's theorem gives the ergodicity. 

\item The Torelli group $\mathcal{I}(S)$ acts preserving the fibers of the fibrations, namely the projectivized spaces of the twisted cohomology group $\mathrm{H}^1_{\alpha}(\Gamma, \mathbb{C})$. This action is in fact projective and thus one gets a nice family of representations of the Torelli group : 

$$ \tau_{\alpha} : \mathcal{I}(S) \longrightarrow \mathrm{PGL}(2g-2, \mathbb{C}) $$ 

\noindent In Section \ref{Torelli}, we provide an explicit computation of the action of a family of Dehn twists along separating curves on $\mathrm{PH}^1_{\alpha}(\Gamma, \mathbb{C})$. We deduce from this computation that for almost all $\alpha$, this action is ergodic.

\end{itemize}

\noindent These two last points together imply the main theorem. A remarkable consequence of the computation is that the mapping class group preserves no symplectic form. In fact it preserves no absolutely continuous measure relatively to the Lebesgue measure, which contrasts with the case where $G$ is reductive, in which we have such a symplectic form at hand, by Goldman's work (see \cite{Gold2}).

\noindent Our original motivation was to study the holonomy of branched affine structures. A direct corollary is that the set of representation arising as the holonomy of such a structure is an open set of full measure of the character variety.

\paragraph*{Acknowledgements.}

 I would particularly like to thank Julien Marché for carefully explaining twisted cohomology to me, Louis Funar for pointing out that the representations of the Torelli group I am considering were originally brought to light by Chueshev, and Serge Cantat for asking me the question that led to this paper. I also thank Luc Pirio for helpful discussions and useful comments on the text.
 
 \noindent I am extremely grateful to my advisor Bertrand Deroin who encouraged me to get involved in mapping class group dynamics. His constant encouragements,
 advice and careful reading of this text made his contribution to this work invaluable to me. 
 \noindent I am also very thankful to an anonymous referee, notably for suggesting a significant improvement of the proof of Proposition 8, but also for various comments on the substance and structure of this paper, which have rendered it much clearer.
 
\noindent  Finally, I would like to thank the DMA at \'Ecole Normale Supérieure which gave me wonderful working conditions.

\newpage

We introduce notations that will be used all along the paper :

\begin{itemize}

\item $S$ is a closed oriented surface of genus $g \geq 2$.

\item $\Gamma$ is the fundamental group of the surface $S$. 

\item $\mathrm{Aff}(\mathbb{C})$ is the group of complex affine transformations of the complex line. 

\item $\mathrm{Mod}(S)$ is the mapping class group of $S$. 

\end{itemize}

\section{Action of the mapping class group on the character variety.}

\label{charactervariety}

\subsection{Structure of the character variety.}

Let us recall the standard presentation for $\Gamma$ : 

$$ \Gamma = \langle a_1,b_1, \cdots, a_g, b_g \ | \ \prod_{i=1}^g{ [a_i,b_i] = 1 }  \rangle $$ 

\noindent Let $ \rho : \Gamma \longrightarrow \mathrm{Aff}(\mathbb{C}) $ be a group homomorphism. If we note $\rho(a_i) : z \mapsto A_iz + U_i$ and $\rho(b_i) : z \mapsto B_iz + V_i$, the following holds : 

$$ \sum_{i=1}^g{ (A_i - 1)V_i + (1 - B_i)U_i} = 0. $$

\noindent Conversely, every set $(A_i,U_i,B_i,V_i) \in \mathbb{C}^{*g}\times \mathbb{C}^g\times \mathbb{C}^{*g}\times \mathbb{C}^g $ verifying the equation above defines a representation of $\Gamma$ in $\mathrm{Aff}(\mathbb{C}) $. Thus $ \mathrm{Hom}(\Gamma, \mathrm{Aff}(\mathbb{C}) ) $ can be seen as an algebraic variety.

\noindent The quotient of $\mathrm{Hom}(\Gamma, \mathrm{Aff}(\mathbb{C}))$ by the action by conjugation of $\mathrm{Aff}(\mathbb{C})$ is not Haussdorf. Nevertheless, the orbits responsible for this are the orbits of representations which are abelian (\textit{i.e.} whose image is an abelian subgroup of $\mathrm{Aff}(\mathbb{C}) )$. Removing these ones, one gets a nice quotient (see Proposition 3).

\begin{definition}

The character variety $ \chi (\Gamma, \mathrm{Aff}(\mathbb{C}) ) $ is defined to be the quotient of $\mathrm{Hom}(\Gamma, \mathrm{Aff}(\mathbb{C})) \setminus \{ \text{abelian} \ \text{representations}\}$ by the action by conjugation of $\mathrm{Aff}(\mathbb{C})$. 

\end{definition}

 Let  $\rho \in \mathrm{Hom}(\Gamma, \mathrm{Aff}(\mathbb{C}))$ be a representation, one can look at its linear part (obtained from $\rho$ just by post composing by the natural group homomorphism $\mathbb{C}^* \ltimes \mathbb{C} \longrightarrow \mathbb{C}^*$). This allows us to define  : 

$$ l : \mathrm{Hom}(\Gamma, \mathrm{Aff}(\mathbb{C})) \longrightarrow \mathrm{Hom}(\Gamma, \mathbb{C}^*) = \mathrm{Hom}(H_1(S, \mathbb{Z}), \mathbb{C}^*) $$ \noindent which factors through $\chi (\Gamma, \mathrm{Aff}(\mathbb{C}) )$, because two conjugate representations have the same linear part.

\begin{proposition}

The map $ L : \chi (\Gamma, \mathrm{Aff}(\mathbb{C}) ) \longrightarrow H^1(S, \mathbb{C}^*) $  is a projective fibration with fiber  $\mathbb{CP}^{2g-3}$.

\end{proposition}

\begin{proof}

The map $l$ restricted to $ \mathrm{Hom}(\Gamma, \mathrm{Aff}(\mathbb{C})) \setminus l^{-1}(\{ 1 \})$ is a vector bundle with fiber $\mathbb{C}^{2g-1}$. Furthermore, for all $\alpha \in H^1(S, \mathbb{C}^*) $, $ l^{-1}(\{ \alpha\}) = Z^1_{\alpha}(\Gamma, \mathbb{C}) $ where 

$$ Z^1_{\alpha}(\Gamma, \mathbb{C}) = \{ \lambda : \Gamma \longrightarrow \mathbb{C} \ | \ \forall \gamma, \gamma' \in \Gamma \ \lambda(\gamma \cdot \gamma') = \lambda(\gamma) + \alpha(\gamma) \lambda(\gamma') \}$$

\noindent The vector space $Z^1_{\alpha}(\Gamma, \mathbb{C})$ is the set of cochains of the cohomology  of $\Gamma$ twisted by $\alpha$. The action of $\mathrm{Aff}(\mathbb{C})$ by conjuguation stabilizes the fibers  $l^{-1}(\{ \alpha\}) = Z^1_{\alpha}(\Gamma, \mathbb{C}) $. Let $\rho := z \mapsto az +b$ and $\lambda \in Z^1_{\alpha}(\Gamma, \mathbb{C})$. We have $\rho \cdot \lambda = b(1 - \alpha) + a \lambda$, so the quotient of $Z^1_{\alpha}(\Gamma, \mathbb{C})$ by the action of $\mathrm{Aff}(\mathbb{C})$ is the projective space of $Z^1_{\alpha}(\Gamma, \mathbb{C}) / \mathbb{C}\cdot (1 - \alpha) = \mathrm{H}^1_{\alpha}(\Gamma,\mathbb{C)}$. 
\vspace{2mm}

\noindent Take $\lambda \in Z^1_{\alpha}(\Gamma, \mathbb{C})$. It is entirely determined by the data of $\lambda(a_1), \lambda(b_1), \cdots$ $\lambda(a_g), \lambda(b_g)$ and those $2g$ complex numbers must satisfy the linear relation 
$$ \sum_{i=1}^g{\lambda(a_i)(1 - \alpha(a_i) )+ \lambda(b_1)(\alpha(b_i) - 1 )} =  0 $$

\noindent Conversely, the data of $2g$ complex numbers satisfying the linear relation above defines an element of $ Z^1_{\alpha}(\Gamma, \mathbb{C}) $. Therefore $ Z^1_{\alpha}(\Gamma, \mathbb{C})$ has complex dimension $2g - 1$ and  $\mathrm{H}^1_{\alpha}(\Gamma, \mathbb{C})$ has complex dimension $2g - 2$. Hence the fiber is isomorphic to $\mathbb{CP}^{2g-3}$.

\end{proof}

\noindent From now on, $\chi$ will be the variety of $\mathrm{Aff}(\mathbb{C})$-characters. 

\vspace{3mm}

\noindent Let $H$ is a subgroup of $\mathbb{C}^*$. We define 

$$ \chi_H = \{ \rho \in \chi \ | \ \mathrm{Im}(L(\rho) ) \subset H \} $$ 

\noindent One will say that a representation $\rho$ is 

\begin{enumerate}

\item \textbf{unitary} (or Euclidean) if it belongs to $\chi_{\mathbb{U}}$, where $\mathbb{U}$ is the set of complex numbers of absolute value $1$.

\item \textbf{real} if it belongs to $ \chi_{\mathbb{R}^*}$.

\item \textbf{almost real} if there exists a subgroup of finite index $\Gamma'$ in $\Gamma$ such that $L(\rho)(\Gamma') \subset \mathbb{R}^*$. 

\item \textbf{abelian} is the image of $\rho$ is an abelian subgroup of $\mathrm{Aff}(\mathbb{C})$.

\item \textbf{strictly affine} in any other case.

\end{enumerate}

\subsection{The $\mathrm{Mod}(S)$ action.}

The mapping class group of a closed surface $S$ is classically defined as

$$\mathrm{Mod}(S) = \mathrm{Homeo}^+(S) / \mathrm{Homeo_0(S)}$$

\noindent Any element of $\mathrm{Mod}(S)$ defines an element of $ \mathrm{Out}(\Gamma) = \mathrm{Aut}(\Gamma) / \mathrm{Inn}(\Gamma) $. By a theorem of Dehn-Nielsen-Baer

$$ \mathrm{Mod}(S) \simeq  \mathrm{Out}^+(\Gamma) $$ 

\noindent where $\mathrm{Out}^+(\Gamma)$ is the subgroup of elements in $\mathrm{Out}(\Gamma)$ preserving the fundamental class in $\mathrm{H}^2(\Gamma, \mathbb{Z})$.

\vspace{2mm}

Notice now that any element of $\mathrm{Aut}(\Gamma)$ acts on $\mathrm{Hom}(\Gamma, \mathrm{Aff}(\mathbb{C}) )$ by precomposition. This action induces an action of $\mathrm{Out}(\Gamma)$ on the character variety. An important remark which will be detailed later is that this action preserves the fiber bundle structure described in the previous section.

\begin{proposition}

\begin{enumerate}

\item Let $H$ be a subgroup of $\mathbb{C}^*$. Then the $\mathrm{Mod}(S)$-action preserves $\chi_H$.

\item  The $\mathrm{Mod}(S)$-action preserves the set of almost-real representations.

\item The $\mathrm{Mod}(S)$-action preserves the set of strictly affine representations.

\end{enumerate}

\end{proposition}

\paragraph*{Remark}

This action preserves no measure \textit{a priori}. Still $\chi$ is a differentiable manifold and even tough the Lebesgue measure is not canonically defined, it makes sense to say that a subset $A$ has measure zero (just say that its Lebesgue measure in any chart is zero). In a more general setting, an action by diffeomorphisms on a manifold will be said to be ergodic if any invariant subset has zero measure or full measure in the sense defined previously.

\subsection{The symplectic representation.}

The mapping class group acts naturally on $\mathrm{H}_1(S,\mathbb{Z)}$, preserving the symplectic intersection form. Up to the choice of a symplectic basis of $\mathrm{H}_1(S,\mathbb{Z)}$, one gets a linear representation of $\mathrm{Mod}(S)$ in $\mathrm{Sp}(2g, \mathbb{Z})$ : 
$$\Psi : \mathrm{Mod}(S) \longrightarrow \mathrm{Sp}(2g, \mathbb{Z}).$$

\noindent Let us denote by $ \mathcal{I}(S)$ the kernel of this representation. This group is usually called the Torelli group. It is the subgroup of $\mathrm{Mod}(S)$ acting trivially on the homology of $S$.

\begin{theorem}

The image of the symplectic representation is  $\mathrm{Sp}(2g, \mathbb{Z})$.

\end{theorem}

\noindent This theorem was originally proved by Poincaré. A modern proof of this theorem can be found in \cite{FM}.

\noindent This way $\mathrm{Mod}(S)$ acts on $\mathrm{Hom}(\mathrm{H}_1(S, \mathbb{Z}), \mathbb{C}^*)$ by precomposition by the image of the symplectic representation. This means that for $f \in \mathrm{Mod}(S)$,  the following diagram commutes :

 \[\xymatrix{
\chi \ar[d]^L \ar[r]^{f} &  \chi \ar[d]^{L}  \\
  \mathrm{H}^1(S,\mathbb{C}^*) \ar[r]_{\Psi(f)}  &  \mathrm{H}^1(S,\mathbb{C}^*)}
\]

\subsection{The Torelli group action on the fibers.}

\begin{proposition}

The Torelli group $\mathcal{I}(S)$ preserves the fibers of $L$, and acts on them by projective transformations.

\end{proposition}

\begin{proof}

Let $f$ be an automorphism whose class in $\mathrm{Mod}(S)$ belongs to $\mathcal{I}(S)$. For any $\alpha \in \mathrm{H}^1(S, \mathbb{C}^*)$, $f$ acts linearly on  $ Z^1_{\alpha}(\Gamma, \mathbb{C})$, preserving the line generated by $1 - \alpha$. Thus $f$ defines a linear automorphism $ \mathrm{H}^1_{\alpha}(S, \mathbb{C})$, and so a projective transformation of $\mathrm{P}\mathrm{H}^1_{\alpha}(S, \mathbb{C})$.

\end{proof}

\section{Ergodicity of the $\mathrm{Sp}(2g, \mathbb{Z})$-action on $(\mathbb{C}^*)^{2g}$.}
\label{ergodicitebase}

The choice of a symplectic basis $a_1, b_1, ..., a_g, b_g$ of $\mathrm{H}_1 (S, \mathbb{Z})$ identifies $ \mathrm{H}^1(S, \mathbb{C}^*)$ and $(\mathbb{C}^*)^{2g} $ via the map 

$$ \alpha \longrightarrow (\alpha(a_1), \alpha(b_1), ..., \alpha(a_g), \alpha(b_g))$$

\noindent The exponential map identifies $\mathbb{T}^{2g} \times \mathbb{R}^{2g}$ with $ \mathrm{H}^1(S, \mathbb{C}^*) \simeq (\mathbb{C}^*)^{2g} $  in such a way that the $\mathrm{Sp}(2g, \mathbb{Z})$-action on $ \mathrm{H}^1_{\alpha}(S, \mathbb{C})$ induces the diagonal action by linear transformations on $\mathbb{T}^{2g} \times \mathbb{R}^{2g}$. Recall the following theorem :
\begin{proposition}

The $\mathrm{Sp}(2g, \mathbb{Z})$-action on $\mathbb{R}^{2g}$ is ergodic.

\end{proposition}

It is a corollary of Moore theorem, which states that if  $\Gamma$ is a lattice in a semi-simple Lie group $G$ and $H$ is a closed non-compact subgroup of $G$, then the $\Gamma$-action on $G/H$ is ergodic. The original proof of this theorem can be found in \cite{MR0193188}.

\begin{proposition}

The $\mathrm{Sp}(2g, \mathbb{Z})$-action on $(\mathbb{C}^*)^{2g}$ is ergodic with respect to the Lebesgue measure.

\end{proposition}

\begin{proof}

Let $B$ a $\mathrm{Sp}(2g, \mathbb{Z})$-invariant measurable subset of $\mathbb{T}^{2g} \times \mathbb{R}^{2g}$ of positive measure and let $A$ be $p^{-1}(B)$ where $p : \mathbb{R}^{2g} \times \mathbb{R}^{2g} \longrightarrow \mathbb{T}^{2g} \times \mathbb{R}^{2g}$ is the natural projection. $A$ is left $\mathbb{Z}^{2g}$-invariant and diagonally $\mathrm{Sp}(2g, \mathbb{Z})$-invariant.

\vspace{3mm}

\noindent $\mathrm{Sp}(2g, \mathbb{R})$ acts transitively on the non-zero level sets of the canonical symplectic form $\omega$ on $\mathbb{R}^{2g}$ and the stabilizer in $\mathrm{Sp}(2g, \mathbb{R})$ of a couple $(x,y) \in \mathbb{R}^{2g} \times \mathbb{R}^{2g}$ is a non-compact closed subgroup of $\mathrm{Sp}(2g, \mathbb{R})$ (see \cite[p.12]{HK} for more details on the structure of the stabilizer). Moore theorem ensures that $\mathrm{Sp}(2g, \mathbb{Z})$ acts ergodically on  $\mathrm{Sp}(2g, \mathbb{R}) / \mathrm{Stab}_{\mathrm{Sp}(2g, \mathbb{R})}\{ (x,y) \} = \omega^{-1}(\{ t \} )$ for all $t \in \mathbb{R}^*$. 

\vspace{3mm}

\noindent Since $A$ is $\mathrm{Sp}(2g, \mathbb{Z})$-invariant, it must be, up to a measurable subset of measure zero, a union of level sets of the symplectic form. Hence $A = \omega^{-1}(I)$ where $I \subset \mathbb{R}$ is a measurable subset. Since $A$ has positive measure, by Fubini theorem $I$ must have positive measure. Let $\alpha $ be any real number and $t \in I$ a density point(which exists according to Lebesgue regularity lemma). Then there exists a couple $(x,y)$ in $\omega^{-1}(\{ t \} )$ and a vector $\vec{k} \in \mathbb{Z}^{2g}$ such that $\omega(\vec{k},y) = \alpha$. Since $A$ is a union of level sets of the symplectic form, $(x,y)$ is a density point of $A$. Translations on the first factor preserve the Lebesgue measure, so $(x + \vec{k}, y)$ is also a density point. By Fubini theorem, $\omega(x + \vec{k}, y) = t + \alpha$ must be a density point of $I$. $\alpha$ has been chosen arbitrarily so $I$ must be equal to $\mathbb{R}$. Hence $A$ is all $\mathbb{R}^{2g} \times \mathbb{R}^{2g}$, therefore the $\mathrm{Sp}(2g, \mathbb{Z})$ action on $\mathbb{T}^{2g} \times \mathbb{R}^{2g}$ is ergodic.
\end{proof}

\section{The Torelli group action on $\mathrm{PH}_{\alpha}^1(\Gamma,\mathbb{C})$ .}
\label{Torelli}
Let us fix once and for all a point $p \in S$ in such a way that we identify $\pi_1(S,p)$ and $\Gamma$. Any diffeomorphism $f$  of $S$ fixing the point $p$ defines canonically an automorphism of $\Gamma$ whose class in $\mathrm{Out}(\Gamma)$ is the class of $f^*$ in $\mathrm{Mod}(S)$. 	In all this section, $\alpha$ is a non-trivial element of $\mathrm{H}^1(S, \mathbb{C}^*)$.

\subsection{Action of a Dehn twist on $\mathrm{H}_{\alpha}^1(\Gamma,\mathbb{C})$}

\begin{proposition}

Any Dehn twist along a separating curve belongs to $\mathcal{I}(S)$.

\end{proposition}

\noindent This is a classical result, whose proof can be found in \cite{FM}.

\noindent We now explain how one can make an effective computation of the action of a Dehn twist along a separating curve.

\begin{lemme}

\label{calcul}

Let $\delta$ be a separating curve in $S$ such that $p \in \delta$, and let $[\delta] \in \pi_1(S,p)$ be a representative of the free homotopy class of $\delta$.  Let $T_{\delta}$ be the Dehn twist along $\delta$. Then there exists  $ \mu \in Z^1_{\alpha}(\pi_1(S,p), \mathbb{C})$ such that for all $[\gamma] \in \pi_1 (S,p)$ and $\lambda \in Z^1_{\alpha}(\pi_1(S,p), \mathbb{C})$

$$  \lambda(T_{\delta}([\gamma])) = \mu([\gamma])\lambda([\delta])  + \lambda([\gamma]) $$

\end{lemme}

\begin{proof}

Let $p \in S$ be the base point of $\pi_1 S$. We assume that all closed curves will be based at $p$, unless explicitly mentioned. Let $[\gamma]$ be a class  in $\pi_1 S$ and $\gamma \in [\gamma]$ such that $\gamma$ intersects $\delta$ transversally. Let $q_1, ..., q_k$ be the intersection points of $\gamma$ and $\delta$ in the order along $\gamma$. Let $q_0=p $.

\noindent Let $\beta_i$ be the closed curve going from $p$ to $q_i$ through $\gamma$ and going through $\delta$ (in the positive direction if $(-1)^{i+1} = 1$ or in the negative sens if $(-1)^{i+1}= -1$) until  $p$. Hence :

$$ T_{\delta}([\gamma]) = [\gamma][\beta_k]^{-1}[\delta]^{(-1)^{k+1}}[\beta_k] \cdots[\beta_2]^{-1}[\delta]^{-1}[\beta_2] [\beta_1]^{-1}[\delta][\beta_1] $$
$$ T_{\delta}([\gamma]) = [\gamma] \prod_{i=k}^1{[\beta_i]^{-1}[\delta]^{(-1)^{i+1}}[\beta_i]} $$

\noindent To see this, one just notices that the image of $\gamma$ through $ T_{\delta}$ is a closed path obtained by following $\gamma$ from $p$ to the first intersection point $q_1$, then following $\delta$ in the positive direction until coming back to $q_1$, then following $\gamma$ between $q_1$ and $q_2$, then following $\delta$ in the negative direction until coming back to $q_2$, etc. This path can be deformed into $ \gamma \beta_k^{-1} \delta^{(-1)^{k+1}}\beta_k \cdots \beta_2^{-1}\delta^{-1}\beta_1 \beta_1^{-1}\delta\beta_1$, adding a path going from $q_i$ to $p$ through $\gamma$ in the negative direction and then coming back to $q_i$ from $p$ in the positive direction, right after each time the path travels across $\delta$ .

\noindent We now compute $\lambda(T_{\delta}([\gamma]))$ 

$$  \lambda(T_{\delta}([\gamma]) = \lambda([\delta]\prod_{i=k}^1{[\beta_i]^{-1}[\delta]^{(-1)^{i+1}}[\beta_i]}) $$ 

\noindent Since $\delta$ is separating, $\alpha([\delta]) = 1$. Using the fact that $\lambda$ is a cocyle (\textit{i.e.} for all classes $[\gamma_1]$, $[\gamma_2] \in \pi_1(S,p)$, $\lambda([\gamma_1] \cdot [\gamma_2]) = \lambda([\gamma_1]) + \alpha([\gamma_1])\lambda([\gamma_2])$), one finds

$$ \lambda(T_{\delta}([\gamma]) = \lambda([\gamma]) +  \lambda([\delta]) \alpha([\gamma]) \cdot \sum_{i=1}^k{(-1)^{i+1}\alpha([\beta_i])}
$$

\noindent and $\mu(\gamma) = \alpha(\gamma)\sum_{i=1}^n{(-1)^{i+1}\alpha^{-1}([\beta_i])}$. It remains to check that $\mu$ is an element of $Z^1_{\alpha}(\pi_1(S,p), \mathbb{C})$ which can be seen by remarking that the image under $T_{\delta}$ of the product of two closed curves is the product of their images and computing using the formulas above.

\end{proof}

\subsection{Action of a subgroup generated by two Dehn twists.}

\begin{figure}[!h]

\label{2courbes}

\centering
   \includegraphics[scale=0.5]{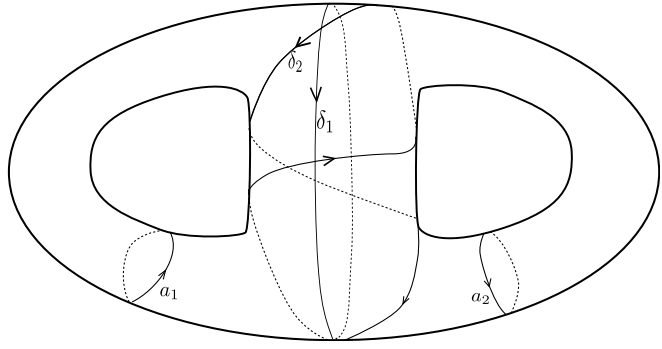}

\caption{The curves $\delta_1$,$\delta_2$, $a_1$ and $a_2$.}   
   
\end{figure}

Let us consider the curves  $\delta_1$ and $\delta_2$ from Figure 1. The Dehn twists along those curves generate a subgroup $G \subset \mathcal{I}(S)$. 

\noindent Let $T_i$ be the  automorphism of $\Gamma$ induced by the Dehn twist along $\delta_i$. $T_i$ acts on $\mathrm{Z}^1_{\alpha}(\Gamma, \mathbb{C)}$ preserving the line generated by $(1 - \alpha)$. Lemma \ref{calcul} ensures  that the left action of $T_i^{-1}$ is  $T_i^{-1} \cdot \lambda = \lambda \circ T_i= \lambda + \varphi_i \cdot \mu_i $ where $\mu_i \in \mathrm{Z}^1_{\alpha}(\Gamma, \mathbb{C)}$ and $\varphi_i $ is the linear form $ \lambda \mapsto \lambda([\delta_i] )$, verifying $\varphi_i(\mu_i) = 0$. $a_1$ and $a_2$ are the curves drawn on Figure $1$.

\begin{proposition}
\label{calculaction}

\noindent 

\begin{enumerate}
\item $ \mu_1(\delta_1) = (1 - \alpha(a_1)^{-1})\cdot(1 - \alpha(a_2)^{-1}) $
\item $ \mu_2(\delta_2) = (1 - \alpha(a_1)) \cdot(1 - \alpha(a_2))  $
\item $ \mu_1(\delta_1) = 0$
\item $ \mu_2(\delta_2) = 0$ 

\end{enumerate}

\end{proposition}

\begin{proof}

\noindent The two last inequalities follow directly  from the fact that a simple closed curve does not self-intersect.

\noindent Write $\mu_1(\delta_2) = \sum_{i=1}^n{\epsilon(i)\alpha^{-1}([\beta_i])}$ according to Proposition \ref{calculaction}. Let us compute the $\beta_i$ using the algorithm described in the proof of Lemma \ref{calcul}.  $\beta_1$ is null-homotopic since $\delta_1$ and $\delta_2$ first intersect at $p$.

\begin{figure}[H]

\label{intersection}

\centering
   \includegraphics[scale=0.6]{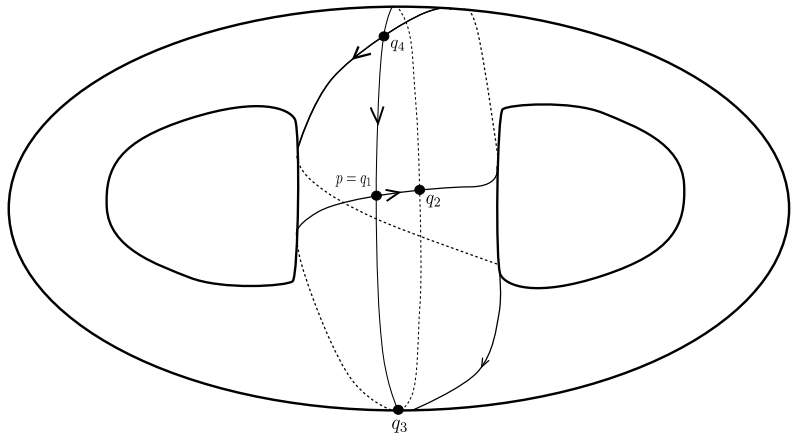}
\caption{Combinatorics of the intersections between $\delta_1$ and $\delta_2$}
   
\end{figure}

\noindent  $\beta_2$ is the curve built following $\delta_2$ from $p$ to $q_2$ then going to $p$ following $\delta_1$. This gives the following curve :

\begin{figure}[H]

\label{intersection}

\centering
   \includegraphics[scale=0.4]{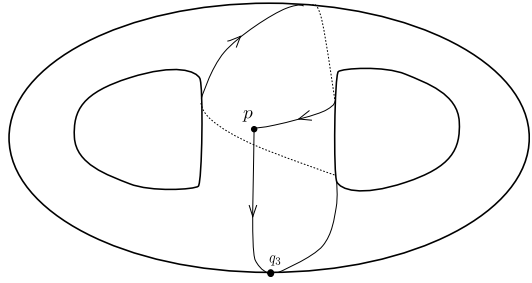}

\caption{The curve $\beta_2$}
   
\end{figure}

\noindent The curve $\beta_2$ is homologuous to $a_1^{-1}$. Proceeding with the algorithm, one finds : 

\begin{itemize}

\item $\beta_1$ is homologuous to $0$.
\item $\beta_2$ is homologuous to $a_1^{-1}$.
\item $\beta_3$ is homologuous to $a_1^{-1}a_2^{-1}$.
\item $\beta_4$ is homologuous to $a_2^{-1}$
\end{itemize}

\noindent This gives $ \mu_1(\delta_2) = 1 -\alpha(a_1) + \alpha(a_1) \alpha(a_2) - \alpha(a_2)$. A likewise calculation gives the value of $ \mu_2(\delta_1)$.

\end{proof}

\begin{proposition}

 $[\mu_1]$ and $[\mu_2]  \in \mathrm{H}_{\alpha}^1(\Gamma,\mathbb{C})$  form a basis of $\mathrm{H}_{\alpha}^1(\Gamma,\mathbb{C})$ for all $\alpha$ in a dense set open set of full measure.

\end{proposition}

\begin{proof}

Assume there exists constants $a,b,c$ such that 

$$ a \mu_1 + b \mu_2 + c(1 - \alpha) = 0 $$ 

\noindent Evaluating on $\delta_1$ and $\delta_2$ , one finds $0 = a \mu_1(\delta_2) = b \mu_2(\delta_1)$.  For $\alpha$ in a dense open set of full measure(the set of $\alpha$ such that  $(1 - \alpha(a_1)^{-1})(1 - \alpha(a_2)^{-1})$ and $ (1 - \alpha(a_1))(1 - \alpha(a_2)) $ do not vanish), $a=b=0$, and so $c=0$.

\end{proof}

\noindent Matrices of $T_1^{-1}$ and $T_2^{-1}$ in this basis are : 

  $$  \left( \begin{array}{cc}
     1 & (1 - \alpha^{-1}(a_1))(1 - \alpha(a_2)^{-1})  \\
      0 & 1   \\
   \end{array} \right) , \left(          \begin{array}{cc}
      1 &  0 \\
      (1 - \alpha(a_1))(1 - \alpha(a_2))   & 1 \\
   \end{array} \right) $$

\subsection{A criterion for ergodicity. }

\begin{lemme}[Jorgensen]

If two matrices $A$ and $B$ generate a non-elementary discrete subgroup of  $\mathrm{PSL}(2, \mathbb{C})$ then 

$$ |\mathrm{Tr}(A)^2 - 4| + |\mathrm{Tr}(ABA^{-1}B^{-1}) - 2 | \geq 1 $$  

\end{lemme}

\noindent This lemma is proven in \cite{MR0427627}. 

\vspace{2mm}

\noindent Let us compute the quantity of the lemma for $A= \left( \begin{array}{cc}
     1 &  a \\
      0 & 1   \\
   \end{array} \right) $ and $B= \left( \begin{array}{cc}
     1 &  0 \\
      b & 1   \\
   \end{array} \right) $.

$$ \mathrm{Tr}(ABA^{-1}B^{-1}) = 2 + (ab)^2 $$
$$ \mathrm{Tr}(A) = 2 $$ 

\noindent So if $A$ and $B$ generate a non-elementary subgroup and if $|ab| < 1$, $\langle A, B \rangle$ is not discrete. One the other hand, it is clear that when $a$ and $b$ are nonzero, the group generated by $A$ and $B$ is non-elementary. In that case, $A$ acts by translations on $\mathbb{CP}^1$, the only point of finite orbit for $A$ is the point at infinity. But since $b \neq 0$, $B$ sends the point at infinity on $0$ which has infinite orbit for the action of $A$.

\begin{proposition}

If  $H$  is a non-discrete and non-elementary subgroup of $\mathrm{SL}(2,\mathbb{C}) $, then $\overline{H}  $  is either all $\mathrm{SL}(2,\mathbb{C}) $ or conjugate to $\mathrm{SL}(2,\mathbb{R})$, a $\mathbb{Z}/2\mathbb{Z}$-extension of  $\mathrm{SL}(2,\mathbb{R})$, $\mathrm{SU}(2)$ or a finite extension of $\mathrm{SU}(2)$. 

\end{proposition}

\noindent This proposition can be found in \cite{MR2553578}(p.69).

\begin{lemme}
\label{ergodique}
Let $H$ be a subgroup of $\mathrm{SL}(n+1,\mathbb{C})$ such that the action of $\overline{H}$ on $\mathbb{CP}^n$ is transitive. Then the action of $H$ on $\mathbb{CP}^n$ is ergodic. 

\end{lemme}

\begin{proof}

This lemma is a consequence of Lebesgue regularity lemma.

\end{proof}

\section{Proof of the main theorem in genus 2.}

The set $U$ of elements $\alpha \in  \mathrm{H}^1(S,\mathbb{C}^*) $ such that  $| (1 - \alpha(a_1))(1 - \alpha(a_2))(1 - \alpha(a_1)^{-1})(1 - \alpha(a_2)^{-1})| < 1$ and $ (1 - \alpha(a_1))(1 - \alpha(a_2))(1 - \alpha(a_1)^{-1})(1 - \alpha(a_2)^{-1}) \notin \mathbb{R}$ has positive measure (it contains an open set of $(\mathbb{C}^*)^4$ with $2$ analytic submanifolds of codimension $1$ removed). According to Proposition 8, the mapping class group action on $\mathrm{H}^1(S,\mathbb{C}^*) \simeq (\mathbb{C}^*)^4$ is ergodic, hence $V = \mathrm{Mod}(S)\cdot U $ has full measure. 

\begin{proposition}

\label{ergodicitetorelli}

For all $\alpha \in V$, the Torelli group action on  $\mathrm{PH}_{\alpha}^1(\Gamma,\mathbb{C})$ is ergodic.

\end{proposition}

\begin{proof}

Consider $\alpha \in V$. Then there exists $\beta \in U$ and $ \phi \in \mathrm{Mod}(S)$ such that $\phi \cdot \beta = \alpha$. Recall that $G \subset \mathcal{I}(S)$ is the group generated by the Dehn twists along $\delta_1$ and $\delta_2$. Precomposing by $\phi$ gives a projective isomorphism :

$$\phi^* :  \mathrm{PH}_{\beta}^1(\Gamma,\mathbb{C}) \longrightarrow  \mathrm{PH}_{\alpha}^1(\Gamma,\mathbb{C})$$ 

\noindent such that the action of the groups $G$ and $\phi G \phi^{-1}$ (on $ \mathrm{PH}_{\beta}^1(\Gamma,\mathbb{C})$ and $\mathrm{PH}_{\alpha}^1(\Gamma,\mathbb{C})$ respectively) are conjugated by $\phi^*$. If $\beta \in U$, the $G$-action on $ \mathrm{PH}_{\beta}^1(\Gamma,\mathbb{C}) \simeq \mathbb{CP}^1$ is the action of a group with identity component of the closure isomorphic to $\mathrm{PSU}(2)$ or $\mathrm{PSL}(2, \mathbb{C})$ (we have assumed that $ (1 - \alpha(a_1))(1 - \alpha(a_2))(1 - \alpha(a_1)^{-1})(1 - \alpha(a_2)^{-1}) \notin \mathbb{R}$, hence the traces of the element of $G$ acting on $\mathbb{CP}^1$ do not all belong to $\mathbb{R}$, according to the computation made above. Hence one can exclude that the closure is isomorphic to  $\mathrm{PSL}(2,\mathbb{R})$ or a $\mathbb{Z}/2\mathbb{Z}$-extension of  $\mathrm{PSL}(2,\mathbb{R})$)  . Lemma \ref{ergodique} ensures that this action is ergodic, so the $\phi G \phi^{-1}$ action on $\mathrm{PH}_{\alpha}^1(\Gamma,\mathbb{C})$ is ergodic since it is conjugated to $G$ through a projective isomporhism.

\end{proof}

One can take as the Lebesgue measure on  $\chi \setminus L^{-1}(\{ 1 \}) $ the measure $m = \mu \otimes \nu_{\alpha}$ where $\mu$ is the Lebesgue measure on $\mathrm{H}^1(S,\mathbb{C}^*) $ and $(\nu_\alpha)_{\alpha \in \mathrm{H}^1(S,\mathbb{C}^*) }$ is a family of measures on $\mathrm{PH}_{\alpha}^1(\Gamma,\mathbb{C})$ depending analytically on $\alpha$. 

\vspace{2mm} 
\noindent We are now ready to end the proof of the main theorem in genus $2$. Let $A$ be a $\mathrm{Mod}(S)$-invariant measurable subset of $\chi \setminus L^{-1}(\{ 1 \}) $. If $\mu(L(A)) = 0$,  then $m(A) = 0$. Thus we can assume $\mu( L(A) ) > 0 $. Since the $\mathrm{Mod}(S)$ action on $\mathrm{H}^1(S,\mathbb{C}^*) $ is ergodic, $L(A)$ has full measure. Put $A_{\alpha} = A \cap \mathrm{PH}_{\alpha}^1(\Gamma,\mathbb{C})$. Fubini theorem implies that 

$$ m(A \cap B) = \int_{L(A \cap B)}{\nu_{\alpha}(A_{\alpha}\cap B)d\mu}  $$ 

\noindent where $B$ is any measurable subset of $\chi \setminus L^{-1}(\{ 1 \}) $.

\noindent If $m(A) > 0$, there exists $\epsilon > 0$ and a set with positive measure $W \subset L(A)$ for which $\forall \alpha \in W$, $\nu_{\alpha}(A_{\alpha}) > \epsilon$. Remind that the set $V$ has full measure so $\mu(W \cap V) > 0$. Since $\mu(W \cap V) > 0$, $\mathrm{Mod}(S)\cdot (W \cap V)$ has full measure. But if $\alpha \in \mathrm{Mod}(S)\cdot ( W \cap V) \subset V$, $\nu_{\alpha}(A_{\alpha}) > 0$ because it contains the image of a $A_{\beta}$ of a map $\phi \in \mathrm{Mod}(S)$ sending $\beta$ on $\alpha$ for a certain $\beta$ in $W \cap V$. But since $\alpha $ belongs to $V$, $\nu_{\alpha}(A_{\alpha}) > 0$  and the Torelli group action on $\mathrm{PH}_{\alpha}^1(\Gamma,\mathbb{C})$ is ergodic, $A_{\alpha}$ has full measure. So for almost all $\alpha$, 
$\nu_{\alpha}(A_{\alpha}\cap B) = \nu_{\alpha}(\mathrm{PH}_{\alpha}^1(\Gamma,\mathbb{C})\cap B)$ and 

$$ m(A \cap B) = m(B) $$

\noindent So $A$ has full measure, which proves that the action is ergodic.

\section{Higher genus.}

We proved in section \ref{ergodicitebase} that the mapping class group action on $\mathrm{H}^1(S,\mathbb{C}^*)$ is ergodic. In genus bigger than $2$, the strategy is still to study the Torelli group action in the fibers  $\mathrm{PH}_{\alpha}^1(\Gamma,\mathbb{C})$. To be more precise, we prove that for almost all $\alpha$, this action is ergodic giving explicit formulas for the action of some specific Dehn twists. 
\noindent Let $p \in S$ be the base point of $\pi_1 S = \Gamma$. Any diffeomorphism $f$ fixing $p$  whose action on  $\mathrm{H}_1(S,\mathbb{Z})$ is trivial acts linearly on $\mathrm{H}_{\alpha}^1(\Gamma,\mathbb{C})$ in such a way that the action of the class of $f$ in $\mathrm{Mod}(S)$ is the projectivized action of $f$ on $\mathrm{PH}_{\alpha}^1(\Gamma,\mathbb{C})$.
\noindent In this section we prove that we can find a subgroup of diffeomorphisms fixing $p$ whose action on $\mathrm{H}_{\alpha}^1(\Gamma,\mathbb{C})$ is ergodic.

\vspace{2mm}

In a way similar to genus $2$, one builds $2g - 2$ curves $(\delta_i, \eta_i)_{1\leq i \leq g-1}$ with the following properties :

\begin{enumerate}

\item For all $i \neq j$,the curve $\delta_i$ (respectively $\eta_i$) is disjoint from the curves $\delta_j$ and $\eta_j$.

\item For a generic $\alpha \in \mathrm{H}^1(S,\mathbb{C}^*)$ (in an open dense subset of full measure), the classes  $[\mu_1], [\nu_1], \cdots , [\mu_{g-1}], [\nu_{g-1}] $ form a basis of $\mathrm{H}_{\alpha}^1(\Gamma,\mathbb{C})$. 

\item Both the action of $T_{\delta_i}$ and $T_{\eta_i}$ stabilize the projective line associated to the plane $[\mu_i], [\nu_i]$. 

\item The group generated by  $T_{\delta_i}$ and $T_{\eta_i}$ acts projectively, the action is ergodic on the stabilized projective line  for all $i$ and for $\alpha $ in an open set. 

\item The $g-1$ groups $G_i = \langle  T_{\delta_i}, T_{\eta_i} \rangle$ commute, this way the $G = G_1 \cdots G_{g-1}$ action is a diagonal action on $\mathbb{C}^{2g-2}  \simeq \mathrm{H}^1_{\alpha}(\Gamma, \mathbb{C})  $.

\end{enumerate}

\noindent Take the genus $2$ surface from Figure 1 and cut it twice along simple closed curves, in a way to get a four holed sphere with boundary :

\begin{figure}[H]

\label{sphere}

\centering
   \includegraphics[scale=0.5]{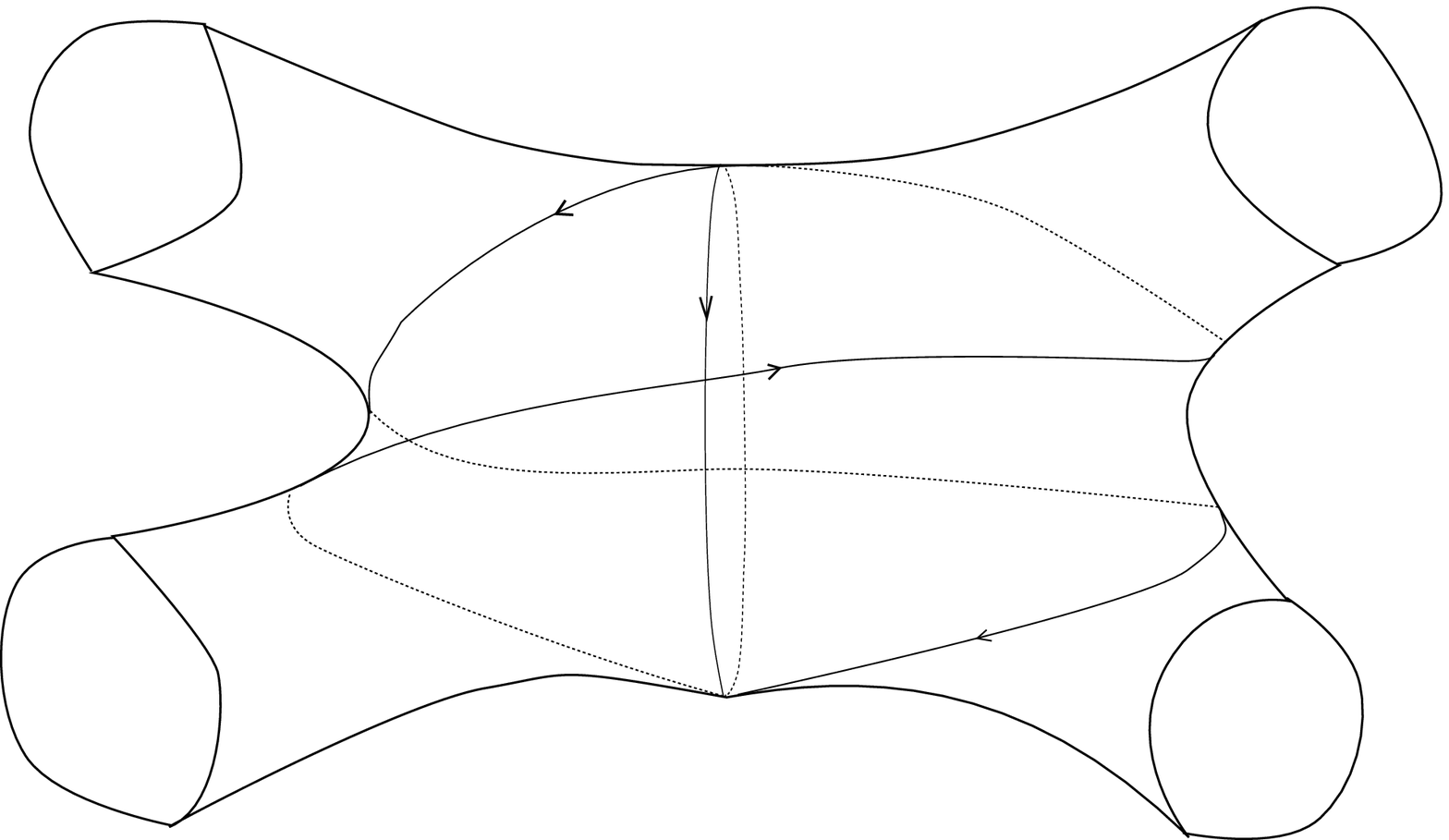}

\end{figure}

\noindent Take $g-1$ copies of this sphere, $S_1$, $S_2$, ..., $S_{g-1}$, each one carrying $2$ marked simple closed curves $\delta_i$ and $\eta_i$. Let us glue them back along the following pattern :

\begin{figure}[H]

\label{sphere}

\centering
   \includegraphics[scale=0.5]{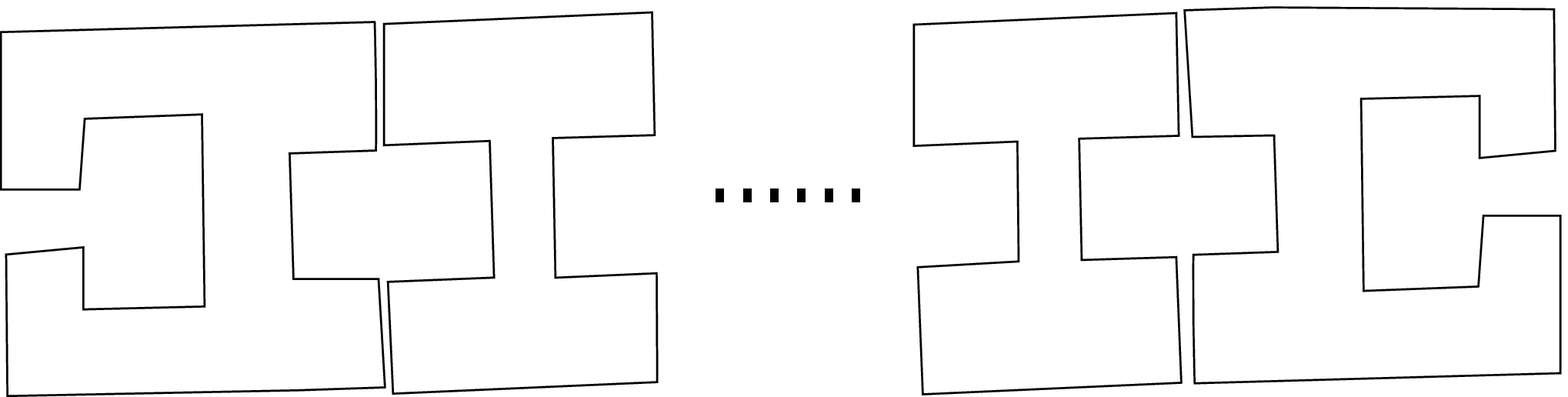}

\end{figure}

This way one gets a genus $g$ surface with the announced family of curves. For $\delta_1$, let $\tilde{\delta_1}$ be the curve built going from $p$ to $\delta_1$ through the chosen path, doing one turn of $\delta_1$ and coming back to $p$. One builds for each $\delta_i$ and $\eta_i$ a curve $\tilde{\delta_i}$ and $\tilde{\eta_i}$ in a similar way.  
\noindent Let $i \neq 1$,  $T_{\delta_i}(\tilde{\delta_1}) = \gamma \tilde{\delta_1} \gamma^{-1}$ for some $\gamma \in \Gamma$ homologuous to $\delta_i$. $\gamma \in \mathrm{D}\Gamma$ since $\delta_i$ is separating, so for all $\lambda \in  \mathrm{H}_{\alpha}^1(\Gamma,\mathbb{C})$, $\lambda(T_{\delta_i}(\tilde{\delta_1})) = \lambda(\tilde{\delta_1})$.

\noindent The same way one can define, associated to $\tilde{\delta_i}, \tilde{\eta_i}$ the cocycles $\mu_i, \nu_i$ such that :

$$ T_{\delta_i}^{-1}\cdot \lambda = \lambda + \lambda(\tilde{\delta_i}) \mu_i$$ 

$$ T_{\eta_i}^{-1}\cdot \lambda = \lambda + \lambda(\tilde{\eta_i}) \nu_i $$

\noindent for all $\lambda \in  \mathrm{H}_{\alpha}^1(\Gamma,\mathbb{C})$.

\vspace{2mm} \noindent Let us assume from now on that $\alpha$ is generic in the following sense: the field generated by the images of $\alpha$ has transcendental dimension $2g$. The set of such $\alpha$ has full Lebesgue measure.

\begin{proposition}
\noindent

\begin{enumerate}

\item

For all $i$, there exist two homology classes $a_i$ and $b_i$ such that \begin{itemize}

\item $ \mu_i(\eta_i) = (1 - \alpha(a_i)) \cdot(1 - \alpha(b_i)) $
\item $ \nu_i(\delta_i) =   (1 - \alpha(a_i)^{-1})\cdot(1 - \alpha(b_i)^{-1}) $
\item $ \mu_i(\delta_i) = 0$
\item $ \nu_i(\eta_i) = 0$ 

\end{itemize}

\item The classes $[\mu_1], [\nu_1], \cdots , [\mu_{g-1}], [\nu_{g-1}] $ span $\mathrm{H}_{\alpha}^1(\Gamma,\mathbb{C})$.

\item For all $ 1 \leq i  \leq g -1 $, the action of the group $G_i$ spanned by $T_{\delta_i}$  and $T_{\eta_i}$ stabilizes the vector space spanned by $[\mu_i]$ and $ [\nu_i]$.

\end{enumerate}
\end{proposition}

\begin{proof}

\noindent 

\begin{enumerate}
\item The first point is exactly Proposition \ref{calculaction} extended to higher genus. The proof works the same way, applying Lemma \ref{calcul}. 

\item One writes a relation of linear dependence :

$$ \sum_i{ u_i \mu_i +  v_i  \nu_i  }= k(1- \alpha) $$ 

\noindent Evaluating in $\tilde{\delta_i}$ and $\tilde{\eta_i}$, one finds that all the coefficients $u_i$ et $v_i$ are zero, which implies $k=0$. 

\item 
Last point is a direct consequence of the remarks above the proposition. If $i \neq j$, then $\mu_i(T_{\delta_i}\tilde{\delta_j}) = \mu_i(\tilde{\delta_j})$, but since $\tilde{\delta_j}$ is homotopic to a curve disjoint from $\delta_i$, $\mu_i(\tilde{\delta_j}) = 0$. It works the same with the curves $\eta_i$, in such a way that the vector space spanned by the $[\mu_i]$ and $[\nu_i]$ is stabilized by the action of $G_i =  \langle  T_{\delta_i}, T_{\eta_i} \rangle$. 

\end{enumerate}

\end{proof}

\begin{figure}[H]

\label{courbes}

\centering
   \includegraphics[scale=0.7]{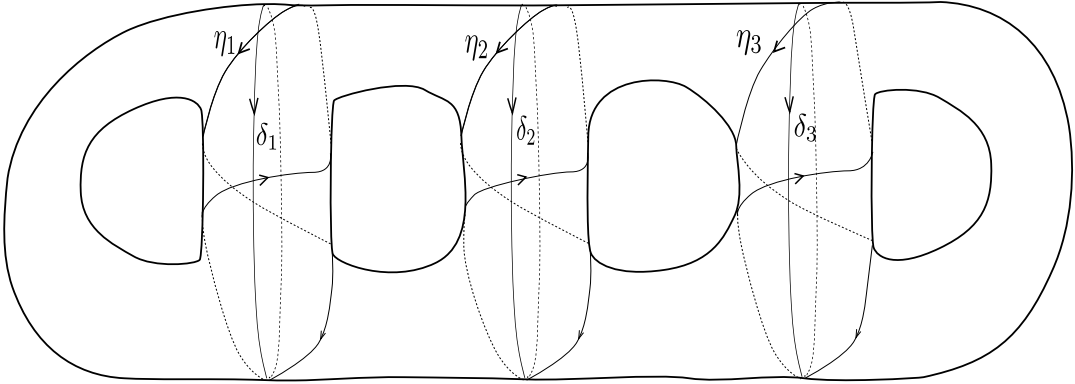}

\caption{The curves $\delta_i$, $\eta_i$ on a genus $4$ surface. }
   
\end{figure}

\noindent We now have everything we need to prove : 

\begin{theorem}
\label{principal}
The action of the mapping class group on $\chi$ is ergodic in genus $g \geq 2$. 

\end{theorem}

\begin{proof}

Let $G$ be the group generated by the $T_{\delta_i}$, $T_{\eta_i}$. $G = G_1 \times \cdots \times G_{g-1} $ since the $G_i$ commute. The $G_i$ action on the vector subspace spanned by $[\mu_i]$ and $[\nu_i]$ is the action of the group spanned by the matrices :

  $$  \left( \begin{array}{cc}
     1 & (1 - \alpha(a_i))(1 - \alpha(b_i))  \\
      0 & 1   \\
   \end{array} \right) , \left(          \begin{array}{cc}
      1 &  0 \\
      (1 - \alpha(a_i)^{-1})(1 - \alpha(b_i)^{-1})   & 1 \\
   \end{array} \right) $$

\noindent Applying Jorgensen's lemma, there exists an open set $U$ of $\mathrm{H}_{\alpha}^1(\Gamma,\mathbb{C})$ for which for all $i$, the action of $G_i$ on the vector space spanned $[\mu_i]$ and $[\nu_i]$ est ergodic (since the action of its closure is transitive). This implies (according to Fubini's theorem) that the action of $G$ on $\mathrm{H}_{\alpha}^1(\Gamma,\mathbb{C})$ is ergodic, hence the action of the Torelli group is ergodic on $\mathrm{PH}_{\alpha}^1(\Gamma,\mathbb{C})$ for $\alpha \in U$. Proposition \ref{ergodicitetorelli} implies it is ergodic on  $\mathrm{PH}_{\alpha}^1(\Gamma,\mathbb{C})$ for $\alpha$ in a dense subset of full measure. Applying Fubini theorem and using the fact that the action of $\mathrm{Mod}(S)$ is ergodic on $\mathrm{H}^1(S,\mathbb{C}^*)$, one finds that the action of $\mathrm{Mod}(S)$ on $\chi$ is ergodic.

\end{proof}

\begin{corollary}
There is no measure in the class of Lebesgue measure on $\chi(\Gamma, \mathrm{Aff}(\mathbb{C}) )$ invariant by the action of the mapping class group. In particular, there is no invariant symplectic form. 

\end{corollary}

\begin{proof}

This follows directly from the fact that for almost all $\alpha \in \mathrm{H}^1(S, \mathbb{C}^*)$, the Torrelli group acts on $\mathrm{PH}_{\alpha}^1(\Gamma,\mathbb{C})$ through elements of $\mathrm{PGL}(\mathrm{H}_{\alpha}^1(\Gamma,\mathbb{C}))$ having attracting fixed points in $\mathrm{PH}_{\alpha}^1(\Gamma,\mathbb{C})$ .

\end{proof}

\section{Euclidean characters.}

Let us look at the action of the mapping class group on $\chi_{\mathbb{U}}$.

\noindent Let $\rho :\Gamma \longrightarrow \mathrm{Aff}(\mathbb{C})$ be a Euclidean representation (whose linear part ranges in the set of complex number of absolute value $1$) . One can naturally associate to $\rho$ a flat $\mathbb{C}$-bundle over $S$ the following way : let $\tilde{S}$ be a universal cover of $S$, $\Gamma$ acts on $\tilde{S} \times \mathbb{C}$  :

$$ \gamma \cdot (x,z) = (\gamma\cdot x, \rho(\gamma)(z) ) $$ 

\noindent The bundle associated to $\rho$ is the quotient $F_{\rho} = \tilde{S} \times\mathbb{C}/ \Gamma $. The foliation $\tilde{S} \times \mathbb{C}$ (whose leaves are the $\tilde{S} \times \{ \cdot \} $) factors through the quotient and defines a flat connection. Note that this construction can be made for any representation $\rho : \Gamma \longrightarrow \mathrm{Homeo}(\mathbb{C})$. 

\vspace{2mm}

\noindent Whenever $\rho$ is Euclidean, one can define a volume form $\mu_x$ , $x\in S$ on the fibers since the standard volume form on $E = \tilde{S} \times \mathbb{C}$ is preserved by the action of $\Gamma$, since $\rho$ is Euclidean. One can define for each $x \in S$  a volume form $\mu_x$ on the fiber over $x$, to get a $2$-form $\omega$ defined on the whole total space. Moreover the form $\omega$ is closed, since it is the form is $dz$ in the coordinates $(x,z)$.

\begin{proposition}

Let $s$ be a section of the bundle $F_{\rho}$.

$$ \mathrm{v}(\rho) = \int_S{s^*\omega} $$

\noindent does not depend on the choice of the section $s$.  It is the volume of the representation $\rho $ .

\end{proposition}

\begin{proof}

$E$ being convex, two sections $s_1$ and $s_2$ are homotopic through $s_t$. Notice that $\int_S{s^*\omega}$ is the volume of the graph of $\rho$. The proposition is a corollary of Stokes theorem applied to the image of the homotopy $s_t$ in $[0,1] \times F_{\rho}$.

\end{proof}

\noindent The volume defines a function $\mathrm{v} : \mathrm{Hom}(\Gamma, \mathrm{Iso}_+(\mathbb{C}) ) \longrightarrow \mathbb{R} $. Let us study the restriction of this function to  $Z^1_{\alpha}(\Gamma, \mathbb{C}) $ for a given $\alpha \neq \mathrm{1}$. The volume of a cocycle $\lambda \in Z^1_{\alpha}(\Gamma, \mathbb{C}) $ is the volume of the associated representation.

\vspace{2mm}

\noindent This form can also be defined in a entirely homological way. If $\alpha$ and $\beta$ are two elements of $\mathrm{H}^1(S,\mathbb{U})$, one can define an algebraic product : 

$$ \wedge : \mathrm{H}^1_{\alpha}(\Gamma, \mathbb{C}) \times \mathrm{H}^1_{\beta}(\Gamma, \mathbb{C}) \longrightarrow \mathrm{H}^2_{\alpha \beta}(\Gamma, \mathbb{C}) $$ 

\noindent where $\mathrm{H}^2_{\alpha \beta}(\Gamma, \mathbb{C})$ is the second group of the cohomology of $\Gamma$ twisted by $\alpha \beta$. $\mathrm{H}^2_{\alpha }(\Gamma, \mathbb{C})  = 0$ as soon as $\alpha \neq 1$.  The bilinear form $$ \begin{array}{ccccc}
\wedge_{\alpha} &  :  & \mathrm{H}^1_{\alpha}(\Gamma, \mathbb{C}) \times  \mathrm{H}^1_{\alpha}(\Gamma, \mathbb{C})   & \longrightarrow & \mathbb{C} \\
 & & (\lambda, \mu ) & \longmapsto & \lambda \wedge \overline{\mu}

\end{array}  $$

\noindent identifying canonically  $\mathrm{H}^2(\Gamma,\mathbb{C})$ and $\mathbb{C}$.  See \cite{DM} for more details (where everything is done is the case of holed spheres, nevertheless it still holds in our setting). 

\begin{proposition}

Take $\alpha \in \mathrm{H}^1(S,\mathbb{C}^*) $

\begin{enumerate}

\item For $\lambda \in Z^1_{\alpha}(\Gamma, \mathbb{C}) $, $\mathrm{v}(\lambda)$ only depends on the class of $\lambda$ in $H^1_{\alpha}(\Gamma, \mathbb{C}) $.

\item The induced function $\mathrm{v} :  H^1_{\alpha}(\Gamma, \mathbb{C})  \longrightarrow \mathbb{R}$ is a non-degenerate Hermitian form.

\item For all $\alpha$ the signature of the form is $(g-1, g-1)$. 
\end{enumerate}

\end{proposition}

\begin{proof}

\begin{enumerate}
\item 
Remark that if $f := az + b \in \mathrm{Aff}(\mathbb{C}) $ , the map

$$ \begin{array}{ccccc}

\Psi & :  & \tilde{S} \times E & \longrightarrow &\tilde{S} \times E \\

 & & (x,z) & \longmapsto & (x, f(z) ) 
\end{array} $$

\noindent  induces an affine isomorphism between the bundles $F_{\rho} $ and $F_{f \rho f^{-1}} $for any representation $\rho$.  From the definition of the forms $\omega$ one gets $$\Psi_* \omega_{\rho} = |a|^2\omega_{f \rho f^{-1}}  $$ Any two representations define the same element in $ H^1_{\alpha}(\Gamma, \mathbb{C})$ if and only if they are conjugated by a translation. In this case, they have the same volume. The formula above ensures that $\mathrm{v}$ is a Hermitian form.

\item The fact that the form is non degenerate is just Poincaré duality in twisted cohomology.

\item Assume $\alpha$ is real. Then conjugation is an order $2$ endomorphism of $\mathrm{H}^1_{\alpha}(\Gamma, \mathbb{C})$ such that $\mathrm{v}(\overline{\lambda}) = - \mathrm{v}(\lambda)$ for every $\lambda \in  \mathrm{H}^1_{\alpha}(\Gamma, \mathbb{C})$. Since $\mathrm{v}$ is non-degenerate, its signature is $(g-1,g-1)$. An argument of connectivity extends the property to arbitrary $\alpha$. To make this work one needs to see that the signature of the form is continuous in $\alpha$. Notice that this form can be seen as the volume form of Euclidean surfaces with branched points. On the open set of those $\alpha$ who can be realized as the linear holonomy of a flat structure with branched points, the signature is continuous since the volume form is continuous. But this set can easily be shown to be all $\mathrm{H}^1(S, \mathbb{U}) \setminus \{ 1 \}$.

\end{enumerate}

\end{proof}

Let $\chi_{\mathbb{U}}^+$(reps. $\chi_{\mathbb{U}}^-$ and $\chi_{\mathbb{U}}^0$) be the subset of $\chi_{\mathbb{U}}$ defined as the set of representations whose volume is positive (reps. negative and null). $\chi_{\mathbb{U}}^+$ and $\chi_{\mathbb{U}}^-$ are invariant subsets of $\chi_{\mathbb{U}}$ under the action of the mapping class group, both of positive measure for the Lebesgue measure on $\chi_{\mathbb{U}}$.

\begin{proposition}

\begin{enumerate}

\item The action of the mapping class group  preserves $\chi_{\mathbb{U}}^+$, $\chi_{\mathbb{U}}^-$ and $\chi_{\mathbb{U}}^0$. 

\item For all $\alpha \in \mathrm{H}^1(S,\mathbb{U})$ different from $\{1\}$, the Torelli group acts on  $\mathrm{PH}^1_{\alpha}(\Gamma, \mathbb{C})$ by transformations belonging to $\mathrm{PU}(\wedge_{\alpha})$.

\end{enumerate}

\end{proposition}

\begin{proof}

Just let a lift of a diffeomorphism to $\tilde{S}$ fixing a base point act on $\tilde{S} \times E$ to see that two representations differing from $f^*$ define the same volume form.

\end{proof}

\paragraph*{The representation of the Torelli group in the case of punctured spheres.}

We have defined a family of representation indexed by $\mathrm{H}^1(S,\mathbb{U})$ of the Torelli group in $\mathrm{PU}(\wedge_{\alpha}) \simeq \mathrm{PU}(g-1,g-1)$. Very little is known about this representation except for the fact that for almost all parameters, its image is not discrete. This family was originally discovered by Chueshev in the early 90's, see \cite{Chu}.
\noindent Now assume that $S$ has a finite number of punctures. One can still build a Hermitian form on $\mathrm{H}^1_{\alpha}(\Gamma, \mathbb{C})$  : Veech shows in \cite{Veech} that the signature of the $\wedge_{\alpha}$ depends on $\alpha$. Moreover, one can pick $\alpha$ in order that $\wedge_{\alpha}$ has signature $(1,n)$. The Torelli group still defines a representation in $\mathrm{PU}(1,n)$. 

\noindent It is an important question in complex hyperbolic geometry to build lattices in the isometry group of complex hyperbolic space. It is natural here to ask if these representations might lead to new constructions of lattices in $\mathrm{PU}(1,n)$.

\section{Link with branched affine structures and open problems.}

The original framework of this work was the study of affine branched  structures, especially their holonomy representations. A complex projective structure on a surface $S$ is an atlas of charts in $\mathbb{CP}^1$ where the transition maps are the restriction of elements in $\mathrm{PSL}(2,\mathbb{C}) = \mathrm{Aut}(\mathbb{CP}^1)$. One can also think of a projective structure as a $(\mathbb{CP}^1, \mathrm{PSL}(2, \mathbb{C}))$-structure in the sense of $(X,G)$-structures defined by Thurston. If $S$ is a surface endowed with a projective structure, one can pull this structure back to its universal cover $\tilde{S}$, in such a way this structure factors through the quotient $S = \tilde{S} / \Gamma$ (meaning that $\Gamma$ acts on $\tilde{S}$ by automorphisms of the projective structure). Since $\tilde{S}$ is simply connected, any projective chart can be fully extended to $\tilde{S}$. This defines a local diffeomorphsim

$$ \mathrm{dev} : \tilde{S} \longrightarrow \mathbb{CP}^1 $$ 

\noindent which is unique up to postcomposition by an element of $\mathrm{PSL}(2, \mathbb{C})$. Since the structure factors trough, there exists a morphism $\mathrm{hol} : \Gamma \longrightarrow \mathrm{PSL}(2, \mathbb{C})$ called the \textit{holonomy} such that for every $\gamma \in \Gamma$ and $x \in \tilde{S}$ we have 

$$ \mathrm{dev}(\gamma \cdot x ) = \mathrm{hol}(\gamma)( \mathrm{dev}(x))$$

\noindent Given $(X,G)$, one might ask what are the group homomorphisms which can arise as the holonomy map of a $(X,G)$-structure.

\paragraph{Translations surfaces and periods of abelian differentials.}

A translation surface is an atlas of charts in $\mathbb{C}$ with transition maps being translations. Since such structures can only arise when $S$ is a torus, one has to allow singularities : a finite set of points can carry a conical structure with angle being a integer multiple of $2\pi$. See \cite{Zorich} for a survey on the subject. The holonomy map of such a structure is a morphism $ \omega : \Gamma \longrightarrow \mathbb{C}$ which factors through $ \omega : \mathrm{H}_1(S,\mathbb{Z}) \longrightarrow \mathbb{C}$ since $\mathbb{C}$ is abelian. In this case, the holonomy problem is totally solved since the 20's (see \cite{Haupt}) by the following theorem : 

\begin{theorem}[Haupt, 1920]

An element $\omega \in \mathrm{H}^1(S,\mathbb{C}) = \mathrm{Hom}(\mathrm{H}_1(S,\mathbb{Z}), \mathbb{C})$ is the holonomy map of a translation surface (or equivalently is the periods of an abelian differential over a Riemann surface) if and only if the two following conditions hold :

\begin{enumerate}

\item $ \mathcal{I}(\omega)\cdot \mathcal{R}(\omega) > 0$, where $\mathcal{I}(\omega)$ and $\mathcal{R}(\omega)$ are respectively the imaginary and real part of $\omega$.

\item If the image of $\omega$ in $\mathbb{C}$ is a lattice $\Lambda$, then 
$$ \mathcal{I}(\omega)\cdot \mathcal{R}(\omega) > \mathrm{vol}(\mathbb{C}/\Lambda) $$ 

\end{enumerate}

\end{theorem}

\noindent A proof of this theorem using mapping class group dynamics has been given in \cite{HK}.

\paragraph*{Holonomy of complex projective structures.}

The holonomy problem is also solved in the case of complex projective structures. Let us recall the theorem due to Gallo, Kapovich and Marden (see \cite{GKM}) :

\begin{theorem}

A group homomorphism $\rho : \Gamma \longrightarrow \mathrm{PSL}(2,\mathbb{C})$ is the holonomy of a complex projective structure if and only if the two following conditions hold : 

\begin{enumerate}

\item $\rho$ lifts to $\mathrm{SL}(2,\mathbb{C})$.

\item The image of $\rho$ is a non-elementary subgroup of $\mathrm{PSL}(2,\mathbb{C})$. 

\end{enumerate}

\end{theorem}

\noindent We also can also allow the projective structure to carry singular points which are locally branched projective coverings. Translation surfaces are particular cases of branched projective structures, whose holonomy lives in the subgroup of translations. In this case the holonomy problem is answered by Haupt's theorem. Now one can look at complex affine structures, which are $(\mathbb{C}, \mathrm{Aff}(\mathbb{C}))$-structures with branched points. 

\paragraph*{Complex (branched) affine structures, holonomy and open problems.}

A complex affine structure is defined to be a Riemann surface $S$ with an non constant holomorphic function 

$$ \mathrm{dev} : \tilde{S} \simeq \mathbb{H} \longrightarrow \mathbb{C} $$ \noindent equivariant with respect to a representation $\rho : \Gamma \longrightarrow \mathrm{Aff}(\mathbb{C})$. One can check that this definition is equivalent to the usual definition with charts and transition maps living in $\mathrm{Aff}(\mathbb{C})$. We ask the following question : which representation $\rho : \Gamma \longrightarrow \mathrm{Aff}(\mathbb{C})$ can be realized as the holonomy map of a branched complex affine structure ? A nice argument of Ehresmann popularized by Thurston ensures that the set of geometric holonomies (which are realized by a branched affine structure) is an open subset of the character variety. Another remark is that whenever a representation can be realized as a holonomy map, its entire orbit under the mapping class group action can also be realized as holonomy maps. Hence we have a nice corollary of Theorem \ref{principal} :

\begin{corollary}

The subset of $\chi(\Gamma, \mathrm{Aff}(\mathbb{C}) )$ consisting of representations which can be realized by a branched complex affine structure is an open set of full measure.

\end{corollary}

\noindent We give here a list of questions arising from the study of these affine structures which seem interesting to the author :

\begin{enumerate}

\item Characterize the representations which are the holonomy of a branched affine structure. 

\item Build explicit models realizing a given holonomy.  

\item Describe more precisely the action of the mapping class group on $\chi$ and $\chi_{\mathbb{U}}^+$. Does there exists an analogous theorem to Ratner's, or is it possible to find an orbit whose closure is not homogeneous  ?

\item Study the dynamics of the directional foliation in the case where the holonomy lies in $\mathbb{R}^* \ltimes \mathbb{C}$. Can phenomena different from those known in the case of translation surfaces happen ?

\item Study the family of representations of the Torelli group  $ \tau_{\alpha} : \mathcal{I}(S) \longrightarrow   \mathrm{PGL}(2g-2,\mathbb{C}) $. For which parameter  $\alpha$ is the image of the representation discrete ? When $\alpha$ is unitary, can one build this way lattices in $\mathrm{PU}(g-1,g-1)$ ? 

\item Explore the case where the singularities are arbitrary. 

\item Study the dynamics of the isoholonomic foliation of the moduli space of branched affine complex structures. Is it ergodic ?

\end{enumerate}

Recall that a strictly affine representation is a nonabelian representation which is not unitary and whose angles of linear parts generate an infinite subgroup of $\mathbb{R}/\mathbb{Z}$. About the holonomy problem, the following conjecture seems reasonable : 

\begin{conjecture}

 Every strictly affine representation is the holonomy of a branched affine structure.

\end{conjecture}

\bibliographystyle{alpha}
\bibliography{biblio}

\end{document}